\documentclass[12pt,article, reqno]{amsart}

\usepackage{amsmath, amssymb, amsfonts, amsthm, latexsym}
\usepackage{graphicx}
\usepackage{tikz}
\usepackage{enumitem}

\usepackage[utf8]{inputenc}

\numberwithin{equation}{section}

\usepackage{titlesec}
\titlelabel{\thetitle.\quad}

\newtheorem{Theorem}{Theorem}[subsection]
\newtheorem{Lemma}[Theorem]{Lemma}

\newtheorem{Claim}[]{Claim}
\newtheorem{Proposition}[Theorem]{Proposition}
\newtheorem{Remark}[Theorem]{Remark}

\newtheorem{Example}[Theorem]{Example}

\newtheorem{Definition}[Theorem]{Definition}
\newtheorem{Notation}[Theorem]{Notation}

\newtheorem{TheoremS}{Theorem}[section]
\newtheorem{LemmaS}[TheoremS]{Lemma}
\newtheorem{CorollaryS}[TheoremS]{Corollary}

\newtheorem{RemarkS}[TheoremS]{Remark}

\def\dist{\operatorname{dist}}

\def\RR{{\mathbb R}} 

\def\ZZ{{\mathbb Z}} 
\def\a{{\mathbf a}}
\def\b{{\mathbf b}}

\def\e{{\mathbf e}}
\def\1{{\mathbf 1}}
\def\x{{\mathbf x}}
\def\v{{\mathbf v}}
\def\y{{\mathbf y}}
\def\O{{\mathcal{O}}}

\def\E{{\mathcal E}}
\def\A{{\mathcal A}}
\def\B{{\mathcal B}}

\def\D{{\Delta}}
\def\G{{\Gamma}}
\def\Lc{{\mathcal L}}

\begin{document}

\title{\L ojasiewicz inequalities \\in a certain class of smooth functions}

\author{Hà Minh Lam}
\address{Institute of Mathematics, Vietnam Academy of Science and Technology, 18 Hoang Quoc Viet, 10307 Hanoi, Vietnam}
\email{hmlam@math.ac.vn}

\author{Hà Huy Vui}
\address{Thang Long Institute of Mathematics and Applied Sciences, Nguyem Xuan Yem Road, Hanoi,
Vietnam}
\email{hhvui@math.ac.vn}
\dedicatory{Dedicated to the memory of professor Stanislaw \L ojasiewicz}
\maketitle

\begin{abstract}
Let $f$ be a germ of a smooth function at the orirgin in $\RR^n.$ We show that if $f$ is Kouchnirenko's nondegenerate and satisfies the so called Kamimoto--Nose condition then it admits the \L ojasiewicz inequalities. We compute the \L ojasiewicz exponents for some special cases. In particular, if $f$ is a germ of a smooth convex Kouchnirenko's nondegenerate function and satisfies the Kamimoto--Nose condition, then all its \L ojasiewicz exponents can be expressed very simply in terms of its Newton polyhedron. 
\end{abstract}

\section{Introduction}
The \L ojasiewicz inequalities were born to solve a problem in analysis \cite{Ho}, \cite{L1}, \cite{L2}. Later, they found many applications and generated many results in other fields (see, for instance, \cite{BDLM}, \cite{BL}, \cite{CKT}, \cite{Gw}, \cite{Ha}, \cite{KMP}, \cite{Kuo}).

In general, \L ojasiewicz inequalities are false for smooth functions.

In this paper we study the \L ojasiewicz inequalities in a special class of smooth functions. The existence of \L ojasiewicz inequalities in some classes of non-analytic functions has been studied in a number of works (see, for instance, \cite{BM}, \cite{Ch}, \cite{F}, \cite{Kur}, \cite{Ha2},  \cite{HNP}).
In \cite{BM}, Biestone and Milman observed that one can use resolution of singularities to obtain the \L ojasiewicz inequalities. These authors introduced a class of smooth functions whose elements have resolution of singularities and so they admit the \L ojasiewicz inequalities. This class of functions satisfies relatively strict conditions, both algebraically and analytically.

To extend Varchenko's result on the asymptotics of oscillatory integrals with an analytic non-degenerate (in the Kouchnirenko sense) phase function, Kamimoto and Nose introduced a condition, under which a smooth non-degenerate function admits a toric resolution of singularities \cite{KN}. 

In this paper, we study the \L ojasiewicz inequalities for non-degenerate smooth fuctions, satisfying the Kamimoto and Nose condition. The paper is organized as follows. In \S 2, we recall  the definition of the Kamimoto and Nose condition (the {\it KN-condition}, for short) and toric modification. In \S 3, we show that if a function $f$ is non-degenerate and satisfies the KN-condition, then $f$ admits the following \L ojasiewicz inequalities:

\noindent
{\bf (\L$_0$)} If $g(\x)$ is a smooth function such that $f^{-1}(0) \subset g^{-1}(0)$ for all $\x$ closed to the origin, then there exist $c,\varepsilon >0$ and a positive number $\alpha$, depending only on $f$, such that 
$$|f(\x)| \ge c|g(\x)|^\alpha, \  \ \forall ||\x || < \varepsilon;$$

\noindent
{\bf (\L$_1$)} There exist $c, \varepsilon >0$ and $\theta \in (0, 1)$ such that 
$$||\nabla f(\x)|| \ge c |f(\x)|^\theta,  \ \ \forall ||\x || <\varepsilon;$$

\noindent
{\bf (\L$_2$)} There exist $c, \varepsilon >0$ and $\mathcal{L}>0$ such that
$$|f(\x)| \ge c \dist (\x, f^{-1}(0))^\mathcal{L} ,  \ \ \forall ||\x ||<\varepsilon,$$
where $\dist (\x, f^{-1}(0))$ denotes the distance from $\x$ to the set $f^{-1}(0).$\\

Put
$$\alpha(f) := \inf\{\alpha \ :\  \mbox{(\L}_0\mbox{) holds} \},\ \ $$
$$\theta(f) := \inf\{\theta \ :\  \mbox{(\L$_1$) holds} \},$$
$$\mathcal{L}(f) := \inf\{\mathcal{L} \ :\  \mbox{(\L$_2$) holds} \}.\ \ $$

 The section 4 will be devoted to the computation of the \L ojasiewicz exponents $\alpha(f),$ $\theta(f)$ and $\Lc (f).$ The subsection 4.1 contains some lemmas. The exponent $\theta(f)$ will be computed in \S 4.2 under condition of partially convenientness of $f$ (Definition \ref{Def4.2.1}). In assuming that $f$ is non-negative, the exponents $\alpha(f)$ and $\Lc (f)$ will be computed respectively in \S 4.3 and \S 4.4. In \S 4.5, we assume that $f$ satisfies both conditions of partially convenientness and non-negativity and show that, in this very special case, all the exponents $\alpha(f)$, $\theta(f)$ and $\Lc (f)$ can be expressed very clearly in terms of the Newton polyhedron of $f$.
 
In Section 5, we consider the case when $f$ is a convex function. It turns out that a convex function is partially convenient and non-negative. Hence, results of \S 4.5 hold true for convex functions.

\section{KN-condition and toric resolution of singularities}

We denote by $\ZZ$ the set of integer numbers, $\ZZ_+$ and $\RR_+$ respectively the subset of all nonnegative number in $\ZZ$ and $\RR.$\\

For $\alpha=(\alpha_1, ..., \alpha_n) \in \ZZ^n_+$, we define
$$\partial^\alpha := \left(\frac{\partial}{\partial x_1}\right)^{\alpha_1}\cdots\left(\frac{\partial}{\partial x_n}\right)^{\alpha_n}$$
$$\alpha ! = \alpha_1 !\cdots\alpha_n ! \qquad \qquad \qquad $$
$$|\alpha| = \alpha_1 + ...+\alpha_n \qquad \qquad \quad$$

We denote by $\langle .,. \rangle$ the standard scalar product in $\RR^n.$

For $(\a, l) \in \RR^n \times \RR$, put
$$H(\a, l) := \{\x \in \RR^n\ | \ \langle \a, \x \rangle =l\},$$
$$H^+(\a, l) := \{\x \in \RR^n \ | \ \langle \a, \x \rangle \ge l\}.$$
Let $f$ be a smooth function on an open neighbourhood of the origin in $\RR^n$ and 
$$T^\infty_0 (f) = \sum^\infty_{|\alpha|=0} c_\alpha \x^\alpha$$
be the Taylor series of $f$ at $0 \in \RR^n.$

The {\it Newton polyhedron} of $f$, denoted by  $\G_+ (f)$, is the convex hull of the set
$$\bigcup_{|\alpha|=0}^\infty \{ \alpha + \RR^n_+ \ |\ c_\alpha \neq 0\}.$$

 A pair $(\a,l) \in \RR^n \times \RR$ is {\it valid} for $\G_+(f)$ if $\G_+(f) \subset H^+ (\a,l).$ Any set of the form $\gamma = \G_+(f) \cap H(\a,l),$ where $(\a, l)$ is a valid pair, will be called a {\it face} of $\G_+(f)$, defined by $(\a, l).$
 
 \begin{Definition}\label{Def2.1} \emph{(\cite{KN}) $\ $ We say that $f$ satisfies the} \it{KN-condition}, \emph{if for any face $\gamma$ of $\G_+(f)$ and any pair $(\a, l)$ defining $\gamma$, $\a = (a_1,...,a_n),$ the limit}
$$\lim_{t \rightarrow 0} \frac{f(t^{a_1}x_1, \cdots, t^{a_n}x_n)}{t^l}$$
\emph{exists for all $\x$ belonging a neighbourhood of the origin.}
\end{Definition}

\begin{Proposition}\label{Prop2.1.2}\emph{(\cite[Proposition 6.3, Lemmas 6.4, 6.5]{KN})} The following two statements are equivalent
\begin{itemize}
\item[(i)] $f$ satisfies the KN-condition;
\item[(ii)] There exists a finite subset $S \subset \G_+(f) \cap \ZZ_+^n$ such that $f$ belongs to the ideal $(\x^\alpha \ :\ \alpha \in S) \E_n,$ generated by monomials $\x^\alpha$, $\forall \alpha \in S$, in the ring $\E_n$ of germs of smooth functions at the origin of $\RR^n.$
\end{itemize}
\end{Proposition}

\begin{Definition}\emph{We say that $f$ is {\it convenient}, if $\G_+(f)$ intersects all the coordinate axes.}
\end{Definition}

\begin{Remark} \label{Rem2.4}\emph{(\cite{KN})} 
\begin{itemize}
\item[(i)] If $f$ is convenient, then $f$ satisfies the KN-condition.
\item[(ii)] If $\gamma$ is a compact face of $\G_+(f)$ then the limit in Definition \ref{Def2.1} always exists and for all $\x$ close to the origin, this limit equals
$$f_\gamma= \sum_{\alpha \in \gamma \cap \ZZ^n_+} c_\alpha \x^\alpha.$$
\item[(iii)] If $f$ is an analytic function or if $f$ belongs to the Denjoy-Carleman quasi-analytic classes, then $f$ satisfies the KN-condition.
\end{itemize}
\end{Remark}

\begin{Example} \emph{(see \cite{KN}) Let $f_k= x_1^2 x_2^2 + x_1^k e^{-1/x_2^2}.$ Then $f_k$ satisfies the KN-condition if and only if $k \ge 2.$}
\end{Example}



\subsection{Toric modification}


Let $\a \in (\RR^n_+)^*,$ where $(\RR^n)^*$ is the dual space of $\RR^n.$

We put 
$$l(\a):= \min\{ \langle \a, \x \rangle \ :\ \x \in \G_+(f)\}$$
and
$$\gamma(\a):= \{ \x \in \G_+(f)\ :\ \langle \a, \x \rangle = l(\a) \}.$$
The vectors $\a, \a'$ from $(\RR^n)^*_+$ are said to be equivalent if $\gamma(\a) =\gamma(\a').$ Then, the closure of an equivalent class can be identified with a rational polyhedral cone $\RR_+\a^1 +...+\RR_+ \a^s,$ where $\a^i \in \ZZ^n_+,$ $i=1, ..., s.$ 

The set $\Sigma_0$ of all the closures of equivalent classes gives us a decomposition of $(\RR^n)^*_+$ on closed convex rational polyhedral cones.

A cone $\RR_+\a^1 +...+\RR_+ \a^s$ is said to be an {\it $s$-simplex} if $\a^1, ..., \a^s$ are linearly independent over $\RR.$ An $s$-simplex cone is {\it  unimodular} if there exist $n-s$ vectors $\a^{s+1}, ..., \a^n \in \ZZ^n_+$ such that det$(a^i_j) =  \pm 1,$ where $\a^i =(a^i_1, ..., a^i_n), i=1, ..., n.$ 

According to \cite{Var} there exists a simplicial subdivision $\Sigma$ of $\Sigma_0$, where all $s$-simplices of $\Sigma$ are unimodular. Let $\Sigma^{(n)}$ denotes the set of all $n$-simplices of $\Sigma.$ For each $\sigma \in \Sigma^{(n)},$ we have $\sigma= \RR_+ \a^1(\sigma) + ...+ \RR_+ \a^n(\sigma),$ $\a^i(\sigma) \in \ZZ^n_+,$ $i=1,...,n.$

Let $\RR^n(\sigma)$ be a copy of $\RR^n.$ We denote $\y_\sigma =(y_{\sigma, 1}, ..., y_{\sigma, n})$ the coordinates in $\RR^n(\sigma).$ We define the map 
$$\pi_\sigma: \RR^n(\sigma) \rightarrow \RR^n \qquad \qquad \qquad \qquad\qquad $$
$$x_k = \prod_{j=1}^n y_{\sigma,j}^{a^j_k(\sigma)},\ k=1,\cdots,n.$$

Let $X_{\G_+(f)} := \displaystyle \bigsqcup_\sigma \RR^n(\sigma)$ be the disjoint union of $\RR^n(\sigma),$ where $\sigma$ runs over the set $\Sigma^{(n)}$. We define an equivalent relation on $X_{\G_+(f)}$ as follows:

If $\y_\sigma \in \RR^n(\sigma)$ and $\y_{\sigma'} \in \RR^n(\sigma ')$, then $\y_\sigma \sim \y_{\sigma'}$  if and only if $\pi_{\sigma'}^{-1} \circ \pi_\sigma(\y_\sigma) = \y_{\sigma'}.$\\

Put 
$$Y_{\G_+(f)} := X_{\G_+(f)} \slash \sim, \mbox{ the factor space}.$$
Then, according to \cite{Var}, we have
\begin{itemize}
\item $Y_{\G_+(f)}$ is a non-singular $n$-dimensional algebraic manifold.
\item The map $\pi: Y_{\G_+(f)} \rightarrow \RR^n$ defined by
$$\pi(\y) = \pi_{\sigma} (\y_\sigma), \mbox{ for } \y \in \RR^n(\sigma)$$
is a proper analytic map.
\end{itemize}

A pair $(Y_{\G_+(f)}, \pi)$ is called the {\it toric modification associated with ${\G_+(f)}$}.

For any compact face of $\G_+(f),$ let $f_\gamma(\x)$ denote the polynomial 
$$f_\gamma(\x) = \sum_{\alpha \in \gamma} c_\alpha \x^\alpha.$$

\begin{Definition}\emph{\cite{Kou} We say that $f$ is} {\it non-degenerate} \emph{if for any compact face $\gamma$ of $\G_+(f)$, the system }
$$ \frac{\partial f_\gamma (\x)}{\partial x_j} = 0,\ \ j=1, \cdots, n$$
\emph{has no solutions in $(\RR \setminus 0)^n.$}
\end{Definition}

Let $I \subset \{1, \cdots, n\}$. Put
$$T_I(\RR^n):=\{\x= (x_1, ..., x_n) \in \RR^n: x_i=0, \ \forall i \in I\}$$
and
$$T_I^* (\RR^n) :=\{ \x=(x_1, \cdots, x_n) \in \RR^n \ :\ x_i=0 \mbox{ if and only if } i \in I\}.$$

\begin{Theorem}\emph{\cite[Theorem 8.10]{KN}} \label{Theo2.8} Let $f$ be a non-degenerate function satisfying the KN-condition. Let $\sigma \in \Sigma^{(n)}, \ \sigma = \RR_+\a^1(\sigma) +...+ \RR_+\a^n(\sigma).$ Then there are a neighbourhood $U$ of $0 \in \RR^n$ and a smooth function $f_\sigma(\y_\sigma)$ on the set $\pi_\sigma^{-1}(U)$ such that, for all $\y_\sigma \in \pi_\sigma^{-1}(U),$ we have
$$\displaystyle f \circ \pi_\sigma(\y_\sigma) = \left(\prod_{j=1}^n y_{\sigma,j}^{l(\a^j(\sigma))}\right) f_\sigma(\y_\sigma),$$
where $l(\a^j(\sigma)) = \min \{\langle \x, \a^j(\sigma)\rangle : \x \in \G_+(f)\}, j=1,...,n,$ and the function $f_\sigma$ satisfies the following conditions:
\begin{enumerate}[label=\upshape(\theTheorem.\arabic*), ref= \theTheorem.\arabic*]
\item\label{2.8.1} $f_\sigma (0) \ne 0;$
\item\label{2.8.2} If $I \subset \{1,...,n\}$ such that $\pi_\sigma (T^*_I(\RR^n(\sigma)))=0,$ then for any point $\b \in T^*_I(\RR^n(\sigma))$ with $f_\sigma(\b) =0,$ there exists $j \in \{1,...,n\} \setminus I$ such that $\frac{\partial f_\sigma (\b)}{\partial y_{\sigma,j}} \neq 0.$
\end{enumerate}
\end{Theorem}


\section{The existence of \L ojasiewicz's inequalities}

For $\sigma = \RR_+\a^1(\sigma) +...+ \RR_+\a^n(\sigma) \in \Sigma^{(n)},$ we put
$$l(\sigma) := \max\{l(\a^i(\sigma)), i=1,...,n\}$$
$$N(\sigma):= \sum_{i=1}^n l(\a^i(\sigma))\qquad \qquad \qquad$$
$$L:= \max\{l(\sigma)\ :\ \sigma \in \Sigma^{(n)}\}\quad$$
$$N:= \max\{N(\sigma)\ :\ \sigma \in \Sigma^{(n)}\}$$

\begin{TheoremS} \label{Theo3.1} Let $f$ be a smooth function on a neighbourhood of the origin in $\RR^n,$ $f(0)=0,$ $\nabla f(0)=0.$ Assume that $f$ is non-degenerate and satisfies the KN-condition. Then the following statements hold true:
\begin{itemize}
\item[(\L$_0$)] If $g$ is a smooth function such that $f^{-1}(0) \subset g^{-1}(0)$ for all points belonging a neighbourhood of the origin, then there exist $c,\varepsilon >0$ such that
$$|f(\x)| \ge c|g(\x)|^L\  \mbox{ if }\  ||\x || < \varepsilon.$$
\item[(\L$_1$)] There exist $c, \varepsilon >0$ such that 
$$||\nabla f(\x)|| \ge c |f(\x)|^{1- \frac{1}{N}} \ \mbox{ if }\ ||\x || <\varepsilon.$$
\item[(\L$_2$)] There exist $c, \varepsilon >0$ such that
$$|f(\x)| \ge c \dist (\x, f^{-1}(0))^N \quad \mbox{ if } \ ||\x ||<\varepsilon.$$
\end{itemize}
\end{TheoremS}
\begin{proof} We first prove the following claims:
\begin{Claim}\label{Claim1} Let $\pi: Y_{\G_+(f)} \rightarrow \RR^n$ be a toric modification  and $\y^0 \in \pi^{-1}(0).$ Then there exist a neighbourhood $V(\y^0)$ of $\y^0$ and a system of smooth local coordinates $\y=(y_1,...,y_n)$ centered at $\y^0$ such that
$$f \circ \pi(\y) = + y^{\rho_1}_1...y^{\rho_n}_n$$
or $$f \circ \pi(\y) = - y^{\rho_1}_1...y^{\rho_n}_n$$
where $\rho =(\rho_1, ..., \rho_n) \in \ZZ^n_+$ such that 
$$|\rho| \le N$$
and 
$$ \max\{\rho_i : i=1,...,n\} \le L.$$
\end{Claim}

\begin{proof}
The proof of Claim \ref{Claim1} follows from Theorem \ref{Theo2.8} and the definitions of $L$ and $N$.
\end{proof}

\begin{Claim}\label{Claim2} Let $\y^0 \in \pi^{-1}(0)$. Then there exist a neighbourhood $V(\y^0)$ of $\y^0$ such that the following statements hold:
\begin{itemize}
\item[(i)] If $g$ is a smooth function on a neighbourhood of the origin in $\RR^n$ and $f^{-1}(0) \subset g^{-1}(0)$ for all $\x$ belonging some neighbourhood of the origin, then there exists $c>0$ such that 
$$|f(\x)| \ge c|g(\x)|^L, \mbox{ if } \x\in \pi(V(\y^0)).$$  
\item[(ii)] There exists $c >0$ such that 
$$\nabla f(\x) \ge c|f(\x)|^{1- \frac{1}{N}}, \mbox{ if } \x\in \pi(V(\y^0)).$$
\item[(iii)] There exists $c >0$ such that
$$|f(\x)| \ge c \dist (\x, f^{-1}(0))^N, \mbox{ if } \x\in \pi(V(\y^0)).$$
\end{itemize}
\end{Claim}

\begin{proof}
By Claim \ref{Claim1}, there exist a neighbourhood $V(\y^0)$ of $\y^0$ and a system of local coordinates $\y=(y_1,...,y_n)$, centered at $\y^0$ such that
$$F(\y):= f \circ \pi_\sigma(\y) = \pm y^{\rho_1}_1...y^{\rho_n}_n.$$



\noindent
{\it Proof of (i): } Let $f^{-1}(0) \subset g^{-1}(0)$ and 
$$G(\y) = g \circ \pi(\y).$$
We can choose $V(\y^0)$ such that 
$$\{\y \in V(\y^0): F(\y)=0\} \subset \{\y \in V(\y^0): G(\y)=0\}$$
and 
$$G(\y) = \left(\prod_{i \in I} y_i\right) G_1(\y),$$
where $I:=\{ i : \rho_i >0\}$ and $G_1(\y)$ is a smooth function on $V(\y^0).$ Then, we have for some $c>0$ 
$$|F(\y)| \ge c|G(\y)|^{\rho_{\max}}\ \mbox{ if } \y \in V(\y^0),$$
where $\rho_{\max} = \max\{\rho_i: i=1,...,n\}.$

Hence, $|f(\x)| \ge c |g(\x)|^{\rho_{\max}}$ for $\x \in \pi_(V(\y^0)).$

Since $\rho_{\max} \le L,$ (i) is proved.\\

\noindent
{\it Proof of (ii):} Let, as above $F(\y)= f\circ \pi(\y) = \pm \y^\rho,$ where $\rho= (\rho_1,..., \rho_n).$ For $i=1,...,n$, put
$$\beta^i:= \begin{cases} (\rho_1,...,\rho_{i-1}, \rho_i-1, \rho_{i+1},...,\rho_n) \mbox{ if } \rho_i >0\\ 0 \mbox{ otherwise}\end{cases}$$
$$N(\rho):= \sum_{i=1}^n \rho_i.$$
Then we have
$$\displaystyle \left( 1- \frac{1}{N(\rho)}
\right) \rho = \sum_{i=1}^n \frac{\rho_i}{N(\rho)}\beta^i$$
This implies that 
$$\displaystyle |F(\y)|^{1- \frac{1}{N(\rho)}} =|\y^\rho|^{1-\frac{1}{N(\rho)}} \le \sum_{i=1}^n \frac{\rho_i}{N(\rho)}|\y|^{\beta^i} =\frac{1}{N(\rho)} \sum_{i=1}^n \Big|\frac{\partial F}{\partial y_i}\Big|$$
for all $\y \in V(\y^0).$

Hence, for some $c_0 >0$, we have
\begin{equation} \label{3.1} ||\nabla F(\y)|| \ge c_0 |F(\y)|^{1-\frac{1}{N(\rho)}}, \mbox{ if } \y \in V(\y^0).
\end{equation}

Let $\y^0 \in \RR^n(\sigma)$ with $\sigma \in \Sigma^{(n)}.$ Let $A_{\pi_\sigma}(\y): \RR^n \rightarrow \RR^n$ denote the linear operator given by the Jacobian matrix of $\pi_\sigma$ at $\y \in V(\y^0).$

Clearly, we can assume that $V(\y^0)$ is a compact neighbourhood. Then,
$$M:=\sup \{||A_{\pi_\sigma}(\y)|| : \y \in V(\y^0)\} < \infty.$$
Let $\nabla f(\pi_\sigma(\y))$ be the vector $\nabla f(\x)$ with $\x = \pi_\sigma (\y).$ We have
$$\nabla F(\y)= A_{\pi_\sigma}(\y) (\nabla f(\pi_\sigma(\y))),$$
and
$$||\nabla F(\y)|| \le ||A_{\pi_\sigma}(\y)||.||\nabla f(\pi_\sigma(\y))|| \le M||\nabla f(\pi_\sigma(\y))||.$$
This, together with (\ref{3.1}) imply that
$$||\nabla f(\x)|| \ge \frac{c_0}{M}|f(\x)|^{1-\frac{1}{N(\rho)}}, \ \mbox{ if } \x \in \pi(V(\y^0))$$
and (ii) is proved.\\

\noindent
{\it Proof of (iii):} One can see that $V(\y^0)$ can be chosen such that for any $\y \in V(\y^0)$ and any pair $\y ', \y '' \in V(\y^0),$ we have  
$$F(\y) = f\circ \pi(\y) =\pm \y^\rho,$$
\begin{equation} \label{3.2}|F(\y)|= |\y^\rho| \ge c_0\dist (\y, F^{-1}(0))^{|\rho|} 
\end{equation}
and
\begin{equation}\label{3.3}
c_1||\pi_\sigma(\y')-\pi_\sigma(\y'')|| \le ||\y'-\y''||
\end{equation}
with some $c_0>0$ and $c_1>0.$

Let $\y \in V(\y^0)$ and $\y^* \in F^{-1}(0) \cap V(\y^0)$ be the point such that 
$$\dist(\y,F^{-1}(0)) = ||\y-\y^*||.$$ Then, it follows from (\ref{3.2}) and (\ref{3.3}) that
$$|f(\x)| \ge  c\dist (\x, f^{-1}(0))^{|\rho|}, \ \forall \x \in \pi_\sigma(V(\y^0))$$
with $c=c_0c_1^{\rho}.$ Now, since $|\rho| \le N,$ (iii), and therefore, Claim \ref{Claim2} is proved.
\end{proof}

Now, we will show that Theorem \ref{Theo3.1} follows from Claim \ref{Claim2}.

Firstly, since $\pi^{-1}(0)$ is a compact subset of $Y_{\G_+(f)}$, there exist the points $\y^1, ..., \y^s \in \pi^{-1}(0)$ and the neighbourhoods $V(\y^j), \y^j \in V(\y^j), j= 1,...,s$ and $\displaystyle \pi^{-1}(0) \subset \bigcup_{i=1}^s V(\y^i)$ such that, by Claim \ref{Claim2}:
$$|f(\x)| \ge c |g(\x)|^{L}$$
$$||\nabla f(\x)|| \ge c |f(\x)|^{1 - \frac{1}{N}}$$
and 
$$ |f(\x)| \ge c \dist (\x, f^{-1}(0))^{N},$$
for all $\displaystyle \x \in \pi\left(\bigcup_{i=1}^sV(\y^i)\right).$

Since $\displaystyle \pi\left(\bigcup_{i=1}^sV(\y^i)\right)$ contains a neighbourhood of the origin in $\RR^n,$ Theorem \ref{Theo3.1} is proved. 
\end{proof}

\begin{CorollaryS}\label{Cor3.2} Let $f$ be a smooth function on a neighbourhood of the origin in $\RR^n,$ $f(0) =0$ and $\nabla f(0) =0.$ Assume that $f$ is convenient and non-degenerate. Then $f$ admits inequalities (\L$_0$), (\L$_1$) and (\L$_2$).
\end{CorollaryS}
\begin{proof}
This follows from Theorem \ref{Theo3.1} and Remark \ref{Rem2.4}.
\end{proof}
\section{Computation of the \L ojasiewicz exponents}

\subsection{Some lemmas}$\ $

Let $\varphi_1$ and $\varphi_2$ be non-negative functions on a neighbourhood of a point $\x_0 \in \RR^n.$ By $\varphi_1 \in \mathcal{O}(\varphi_2)$ we mean that there exists $c>0$ such that 
$$c \varphi_1(\x) \le \varphi_2(\x)$$ holds for all $\x$ belonging to some neighbourhood of $\x_0.$ If $\varphi_1 \in \mathcal{O}(\varphi_2)$  and $\varphi_2 \in \mathcal{O}(\varphi_1)$ then we write 
$$\varphi_1 \asymp \varphi_2.$$
By $\varphi_1 \in o(\varphi_2)$ we mean the following: there exists a neighbourhood $U(\x_0)$ of $\x_0$ such that 
$$\varphi_2^{-1}(0) \cap U(\x_0) \subset \varphi_1^{-1}(0) \cap U(\x_0)$$
and for any sequence $\{\x_k\} \subset U(\x_0),$ such that 
$\x_k \rightarrow \x_0$ and $\varphi_2(\x_k) \ne 0\  \forall k,$ $$ \displaystyle \lim_{k \rightarrow \infty} \frac{\varphi_1(\x_k)}{\varphi_2(\x_k)} =0.$$

Let $\A \subset \ZZ^n_+.$ Let $\G_+(\A)$ be the convex hull of the set $\displaystyle \bigcup_{\alpha \in \A} \{\alpha + \RR^n_+\}.$

Let $V_{\G_+(\A)}$ denote the set of all vertices of $\G_+(\A).$

Put $$g_{\G_+(\A)} (\x) := \sum_{\alpha \in V_{\G_+(\A)}} |\x^\alpha|.$$

Let $\G(\A)$ denote the union of all compact faces of ${\G_+(\A)}$.

\begin{Lemma}\label{Lem4.1.1} We have 
\begin{itemize}
\item[(i)] If $\alpha \in {\G_+(\A)} \cap \ZZ^n_+.$ then $|\x^\alpha| \in \O(g_{\G_+(\A)})$
\item[(ii)] If $\alpha \in \left({\G_+(\A)} \cap \ZZ^n_+ \right) \setminus \G(\A),$ then
$$|\x^\alpha| \in o(g_{\G_+(\A)}).$$
\end{itemize}
\end{Lemma}

\begin{proof}
Let $\pi: Y_{\G_+(\A)} \rightarrow \RR^n$ be the toric modification associated with ${\G_+(\A)}$. Let $\y^0$ be an arbitrary point of $\pi^{-1}(0).$ Then $\y^0 \in \pi_\sigma^{-1}(0),$ for some 
$\sigma = \RR_+ \a^1(\sigma) + ...+ \mathbb{R}_+ \a^n(\sigma) \in \Sigma^{(n)}.$
We put 
$$\alpha_\sigma:= \bigcap_{i=1}^n \gamma (\a^i(\sigma)),$$
where $\gamma(\a^i(\sigma))$, $i=1,...,n,$ denotes the face of $\G_+(\A),$ defined by the pair $(\a^i(\sigma), l(\a^i(\sigma))$. Clearly, $\alpha_\sigma \in V_{\G_+(\A)}.$

In order to prove the lemma, it is enough to prove the following: there exists a neighbourhood $V(\y^0)$ of $\y^0$ such that
\begin{itemize}
\item If $\alpha \in \G_+(\A) \cap \ZZ^n_+$ then, in a neighbourhood $V(\y^0)$ of $\y^0$ we have
\begin{equation}\label{4.1}
|\x^\alpha \circ \pi_\sigma(\y_\sigma)| \in \O \left(|\x^{\alpha_\sigma} \circ \pi_\sigma(\y_\sigma)|\right)
\end{equation}
\item Moreover, if  $\alpha \in (\G_+(\A) \cap \ZZ^n_+) \setminus \G (\A),$ then, in $V(\y^0)$ we have
\begin{equation}\label{4.2} |\x^\alpha \circ \pi_\sigma(\y_\sigma)| \in o\left(|\x^{\alpha_\sigma} \circ \pi_\sigma(\y_\sigma)|\right)
\end{equation}
\end{itemize}

Firstly, we assume that $\y^0 =0 \in \RR^n(\sigma).$ Then we have
$$|\x^{\alpha_\sigma} \circ \pi_\sigma(\y_\sigma)| = \Big| \prod_{i=1}^n y_{\sigma,i}^{l(\a^i(\sigma))}\Big|$$
and for any $\alpha \in \G_+(f) \cap \ZZ^n_+,$ $\alpha \ne \alpha_\sigma,$
$$|\x^{\alpha} \circ \pi_\sigma(\y_\sigma)| \in o\left( \Big| \prod_{i=1}^n y_{\sigma,i}^{l(\a^i(\sigma))}\Big|\right).$$
Thus, (\ref{4.1}) and (\ref{4.2}) hold if $\y^0 =0 \in \RR^n(\sigma).$

Now, assume that $\y^0 \in \pi^{-1}_\sigma(0),$ but $\y^0 \ne 0 \in \RR^n(\sigma).$ Then, there exists a proper subset $I$ of the set $\{1,...,n\}$ such that
$$\y^0 \in T^*_I(\RR^n(\sigma)),$$
where $T^*_I(\RR^n(\sigma))$ denotes the set of all the points $\y_\sigma =(y_{\sigma,1}, ..., y_{\sigma,n})$ of $\RR^n(\sigma)$ such that $y_{\sigma,j} =0$ if and only if $j \in I$ and $\pi_\sigma(T^*_I(\RR^n(\sigma)))=0.$

By Lemma 3.7 of \cite{FY}, the set 
$$\gamma:= \bigcap_{i\in I} \gamma(\a^i(\sigma))$$
is a compact face of $\G_+(\A).$

Then, it is easy to see that
$$|\x^{\alpha_\sigma} \circ \pi_\sigma(\y_\sigma)| \asymp \Big| \prod_{i\in I} y_{\sigma,i}^{l(\a^i(\sigma))} \Big|$$
and 
$$|\x^{\alpha_\sigma} \circ \pi_\sigma(\y_\sigma)| \in o\left(\Big| \prod_{i\in I} y_i^{l(\a^i(\sigma))} \Big|\right)$$
if $ \alpha \in \left(\G_+(\A) \cap \ZZ^N_+\right) \setminus \gamma.$

Thus, (\ref{4.1}) and (\ref{4.2}) hold, and the lemma is proved.

\end{proof}
 The following result is an analogue of one result of \cite{FY}, \cite{Y}, where the authors considered the analytic case.
\begin{Lemma}\label{Lem4.1.2}
Let $f$ be a smooth function on a neighbourhood of the origin in $\RR^n$, $f(0)=0$ and $\nabla f(0)=0.$ Assume that $f$ is non-degenerate and satisfies the KN-condition. Then we have
$$\sum_{i=1}^n \Big|x_i \frac{\partial f(\x)}{\partial x_i} \Big| \asymp g_{\G_+(f)} (\x).$$
\end{Lemma}

\begin{proof} 


In order to prove the lemma, it is enough to show that for any $\sigma \in \Sigma^{(n)}$ and any $\y^0 \in \pi^{-1}_\sigma(0),$ there exists a neighbourhood of $\y^0$, in which we have
\begin{equation} \label{4.3}
|\x^{\alpha_\sigma} \circ \pi_\sigma(\y_\sigma)| \asymp \left(\sum_{i=1}^n \Big|x_i \frac{\partial f}{\partial x_i} \Big|\right) \circ \pi_\sigma (\y_\sigma)
\end{equation}
where $\displaystyle \alpha_\sigma := \bigcap_{i=1}^n \gamma(\a^i(\sigma)) \in V_{\G_+(f)}.$

Since $\y^0 \in \pi^{-1}_\sigma(0),$ there exists $I \subset \{1,...,n\}$  such that
$$\y^0 \in T^*_I( \RR^n(\sigma)) \subset \pi^{-1}_\sigma(0).$$ 
Let $\displaystyle \gamma := \displaystyle \bigcap_{i \in I} \gamma(\a^i(\sigma))$ (in particular, $\gamma = \alpha_\sigma$ if $I=\{1,...,n\}).$

Since $\y^0 = (y^0_{\sigma,1},..., y^0_{\sigma,n}),$ where $y^0_{\sigma,j} \ne 0$ if and only if $j \notin I,$ we have
$$\displaystyle |\x^{\alpha_\sigma} \circ \pi_\sigma(\y_\sigma)| \asymp \Big| \prod_{i\in I} y^{l(\a^i(\sigma))}_{\sigma,i}\Big|,$$
in a neighbourhood of $\y^0.$ 

So, to prove (\ref{4.3}), we have to prove that in a neighbourhood of $\y^0,$ we have
\begin{equation}\label{4.4} \displaystyle \left( \sum_{i=1}^n \Big|x_i \frac{\partial f}{\partial x_i}\Big|\right) \circ \pi_\sigma(\y_\sigma) \asymp \Big| \prod_{i \in I} y^{l(\a^i(\sigma))}_{\sigma,i} \Big|.\end{equation}

Since $f$ satisfies the KN-condition, by Proposition \ref{Prop2.1.2}(ii), $f$ can be written in the form 
$$f(\x) = \sum_{\alpha \in S} \x^\alpha \varphi_\alpha(\x),$$
where $S$ is a finite subset of $\G_+(f) \cap \ZZ^n_+$ and, for any $\alpha,$ $\varphi_\alpha$ is a smooth function on a neighbourhood of the origin. 

It is not difficult to see that
\begin{itemize}
\item If $\alpha \in \gamma \cap \ZZ^n_+$ and $\partial^\alpha f(0) \ne 0,$ then $\alpha \in S$.
\item If $\alpha \in S \cap \gamma$ and $\partial^\alpha f(0) = 0,$ then $\displaystyle \varphi_\alpha(0) =0.$
\end{itemize}
These facts show that we can rewrite $f(\x)$ in the form 
\begin{equation}\label{4.5}
f(\x) = f_\gamma(\x) + \sum_{\beta\in S'}\x^\beta \varphi_\beta,
\end{equation}
where $S'\subset \left(\G_+(f) \cap \ZZ^n_+\right)\setminus \gamma.$

It is easy to see that
\begin{equation}\label{4.6}
\Big| x_i \frac{\partial}{\partial x_i}\left(\sum_{\beta \in S'} \x^\beta \varphi_\beta \right)\Big|\circ \pi_\sigma(\y_\sigma) \in o\left(\prod_{i\in I} y_{\sigma, i}^{l(\a^i(\sigma))}\right)
\end{equation}
for any $i=1, ...,n.$

It follows from the definitions of $\gamma$ and $f_\gamma$ that, for any $i \in \{1,...,n\}:$
\begin{equation}\label{4.7}
\displaystyle \Big|x_i \frac{\partial f_\gamma}{\partial x_i}\Big| \circ \pi_\sigma(\y_\sigma) =  \Big| \prod_{i\in I} y_{\sigma, i}^{l(\a^i(\sigma))}\Big|.\left(\Big| x_i \frac{\partial f_\gamma}{\partial x_i}\Big|\circ \pi_\sigma(\hat{\y}_\sigma)\right),
\end{equation} 
where $\hat{\y}_\sigma := (\hat{y}_{\sigma,1}, ..., \hat{y}_{\sigma,n})$ with
$$\hat{y}_{\sigma,j} := \begin{cases} y_{\sigma,j}& \mbox{ if } j \notin I\\ 1 & \mbox{ otherwise}
\end{cases}.$$
Since $\hat{\y}_\sigma \in (\RR(\sigma) \setminus 0)^n$, $\x_\sigma:=\pi_\sigma(\hat{\y}_\sigma) \in (\RR\setminus 0)^n.$ Then, it follows from non-degenerency of $f$, there exists $i_0$ such that
\begin{equation}\label{4.8}
\displaystyle \Big|x_{i_0} \frac{\partial f_\gamma}{\partial x_i}\Big| \circ \pi_\sigma(\hat{\y}_\sigma) \ne 0
\end{equation}
for all $\hat{\y}_\sigma$ belonging to a neighbourhood of $\hat{\y}^0,$ where the point $\hat{\y}^0$ is defined via $\y^0$ by the following way: if $\y^0 = (y^0_1, ..., y^0_n)$ then
$\hat{\y}^0 := (\hat{y}^0_1, ..., \hat{y}^0_n)$ where
$$\hat{y}^0_i := \begin{cases} y^0_i& \mbox{ if } i \notin I\\ 1 & \mbox{ otherwise}
\end{cases}.$$

Thus, (\ref{4.4}) follows from (\ref{4.5}) and (\ref{4.8}). The lemma is proved.
\end{proof}

\begin{Lemma}\label{Lem4.1.3} Let $f$ be a smooth function on a neighbourhood of the origin in $\RR^n,$ $f(0)=0$ and $\nabla f(0)=0.$ Assume that $f$ is non-degenerate, non-negative and satisfies the KN-condition. Then in some neighbourhood of the origin, we have
$$f(\x) \asymp g_{\G_+(f)}(\x).$$
\end{Lemma}

\begin{proof} Let $\sigma = \RR_+\a^1(\sigma) + ...+ \RR_+\a^n(\sigma) \in \Sigma^{(n)},$

$\ \displaystyle \quad \qquad \alpha_\sigma := \bigcap_{i=1}^n \gamma(\a^i(\sigma)),$

and $\ \quad \y^0 \in \pi^{-1}_\sigma(0).$

According to Theorem \ref{Theo2.8}, there exists a smooth function $f_\sigma$ in some neighbourhood of $\y^0$ such that
$$f \circ \pi_\sigma(\y_\sigma) = \left( \prod_{i=1}^n y_{\sigma,i}^{l(\a^i(\sigma))}\right) f_\sigma(\y_\sigma),$$
such that the condition (\ref{2.8.1}) and (\ref{2.8.2}) hold true. 

Since, by (\ref{2.8.1}), $f_\sigma(0) \ne 0$ and $f \ge 0,$ all the numbers $l(\a^i(\sigma)),$ $i= 1,...,n,$ must be even and $f_\sigma(0) >0.$ Then, it implies that, in a neighbourhood of $\y^0 =0 \in \RR^n(\sigma),$ we have
\begin{equation}\label{4.9}f \circ \pi_\sigma(\y_\sigma) = \left( \prod_{i=1}^n y_{\sigma,i}^{l(\a^i(\sigma))}\right) f_\sigma(\y_\sigma) \asymp \prod_{i=1}^n y_{\sigma,i}^{l(\a^i(\sigma))} = \x^{\alpha_\sigma} \circ \pi_\sigma(\y_\sigma).
\end{equation}
Next, we will show that (\ref{4.9}) is still true for a neighbourhood of $\y^0,$ when $\y^0 \in T^*_I(\RR^n(\sigma))$, with $\pi_\sigma(T^*_I(\RR^n(\sigma)))=0$ and $I \ne \{1,...,n\}.$ To do this, it is enough to prove that $f_\sigma(\y^0) \ne 0.$ Assume, by contradiction, that $f_\sigma(\y^0)=0.$ Then, by (\ref{2.8.2}), there exist $j \notin I$ such that $\displaystyle \frac{\partial f_\sigma(\y^0)}{\partial y_j} \ne 0.$ Therefore, the sign of $f_\sigma(\y_\sigma)$ is not constant in a neighbourhood of $\y^0,$ which is impossible because $f\circ \pi_\sigma(\y_\sigma) \ge 0$ and $\displaystyle \prod_{i=1}^n y_{\sigma,i}^{l(\a^i(\sigma))} \ge 0.$

Thus, (\ref{4.9}) holds true for any $\sigma \in \Sigma^{(n)}$ and any $\y^0 \in \pi^{-1}(0).$ This finishes the proof of Lemma \ref{Lem4.1.3}.

\end{proof}

\subsection{Computation of $\theta (f)$}$\ $

We introduce the following notion
\begin{Definition}\label{Def4.2.1}\emph{Let $f$ be a smooth function on a neighbourhood of the origin, $f(0)=0,$ and $\nabla f(0)=0.$ We say that $f$ is {\it partially convenient}, if there exist $J \subset \{1,...,n\}$ and positive integer numbers $\nu_i,$ $i \in J$ such that}
\begin{enumerate}[label=(\roman*)]
\item \emph{$$\{\nu_i \e_i, i \in J\} \subset V_{\G_+(f)},$$ where $\e_i = (0,...,0,\underset{i}{1},0,...,0)$ is the i$-th$ unit vector of $\RR^n;$}
\item \emph{If $\alpha =(\alpha_1, ..., \alpha_n) \in V_{\G_+(f)},$ then 
$$\alpha_j =0 \ \forall j \notin J.$$
}
\end{enumerate}
\end{Definition}
\begin{Remark}\begin{enumerate}[label=\upshape(\roman*), ref=\theTheorem\roman*]
\item \label{4.2.2(i)} \emph{If $f$ is convenient, then $f$ is partially convenient ($J~=~\{1,...,n\}$).}
\item \emph{Let $f(x_1, x_2, x_3) = x_1^4 + x_1x_2 + x_2^4 + x_1^4 x_3^6.$ Then $f$ is partially convenient ($I = \{1, 2\}$), but not convenient.
}
\end{enumerate}
\end{Remark}

We put $\nu(f) := \max\{\nu_i : \nu_i \e_i \in 
V_{\G_+(f)}\}.$

\begin{Theorem}\label{Thm4.2.3} Let $f$ be a smooth function on a neighbourhood of the origin in $\RR^n$, $f(0) =0,$ $\nabla f(0) =0.$ Assume that $f$ is non-denenerate, partially convenient and satisfies the KN-condition. Then we have
$$\displaystyle \theta(f) = 1 -\frac{1}{\nu(f)}.$$
\end{Theorem}

\begin{proof}
Since $f$ satisfies the KN-condition, by Proposition \ref{Prop2.1.2}, we have
\begin{equation}\label{4.10}
f(\x) = \sum_{\alpha \in S} \x^\alpha \varphi_\alpha(\x),
\end{equation}
where $S$ is a finite subset of $\G_+(f) \cap \ZZ^n_+$ and $\varphi_\alpha$, $\alpha \in S,$ is a smooth function in some neighbourhood of the origin.

Then it follows from Lemma \ref{Lem4.1.1} that
\begin{equation}\label{4.11}
|f(\x)| \in \O (g_{\G_+(f)}).
\end{equation}

By Lemma \ref{Lem4.1.2}, we have
\begin{equation}\label{4.12}
\sum_{i=1}^n \big|x_i \frac{\partial f}{\partial x_i} \big| \asymp g_{\G_+(f)}.
\end{equation}
Let $J$ denote the subset of $\{1,...,n\},$ given in Definition \ref{Def4.2.1}. Using (\ref{4.10}), Lemma \ref{Lem4.1.1}(ii), it is not difficult to see that, for any $j \notin J,$ we have
$$\displaystyle \Big| x_j \frac{\partial f }{\partial x_j}\Big| \in o\left(g_{\G_+(f)}\right).$$
This, together with (\ref{4.12}) give us
$$\sum_{i\in J} \big|x_i \frac{\partial f}{\partial x_i} \big| \asymp g_{\G_+(f)},$$
which implies that
\begin{equation}\label{4.14}
g_{\G_+(f)} \in \O\left(||\x_J||. ||(\nabla f)_J||\right),
\end{equation}
where, by definition
$$\displaystyle ||\x_J|| = \left(\sum_{i \in J} |x_i|^2\right)^{1/2}$$
and $$||(\nabla f)_J|| = \displaystyle \left[\sum_{i \in J} \left(\frac{\partial f}{\partial x_i}\right)^2\right]^{1/2}.$$

It is easy to see that 
\begin{equation}\label{4.15}
||\x_J||^{\nu(f)} \in \O(g_{\G_+(f)}).
\end{equation}
It follows from (\ref{4.14}) and (\ref{4.15}) that, for all $\x$ belonging a neighbourhood of the origin, we have
\begin{equation}\label{4.16}
||\x_J|| \in \O\left(||(\nabla f)_J||^{\frac{1}{\nu(f)-1}}\right).
\end{equation}
Combining (\ref{4.14}) and (\ref{4.16}), we get 
$$\displaystyle \left( g_{\G_+(f)}\right)\in \O\left(||(\nabla f)_J||^\frac{\nu(f)}{\nu(f)-1} \right)$$
or, 
$$\displaystyle g_{\G_+(f)}^{1-\frac{1}{\nu(f)}} \in \O(||(\nabla f)_J||),$$
which implies that 
$$\displaystyle g_{\G_+(f)}^{1-\frac{1}{\nu(f)}} \in \O(||\nabla f||).$$
Now, by (\ref{4.11}), we have:
$$\displaystyle |f|^{1-\frac{1}{\nu(f)}}\in \O(||\nabla f||),$$
which means that
$$\theta (f) \le 1- \frac{1}{\nu(f)}.$$

Now, let $\nu_{i_0} = \nu(f),$ with $i_0 \in J.$ 

Restricting $f$ and $g_{\G_+(f)}$ on the line
$$L:= \{ \x=(x_1,...,x_n) \in \RR^n : x_i = 0 \ \forall i \ne i_0\},$$
one can see that
$$\theta (f) \ge 1 - \frac{1}{\nu(f)}.$$
The theorem is proved.
\end{proof}

\subsection{Computation $\alpha(f)$}
\begin{Notation} \emph{For $\alpha = (\alpha_1, ..., \alpha_n) \in \ZZ^n_+,$ we put}
$$I(\alpha):= \{ i: \alpha_i \ne 0\} \qquad \qquad \qquad \qquad \qquad \qquad \qquad \quad $$
$$\hat{\alpha}:= (\hat{\alpha}_1, ..., \hat{\alpha}_n), \mbox{ where } \hat{\alpha}_i=\begin{cases} 1 \mbox{ if } i \in I(\alpha)\\ 0 \mbox{ otherwise} \end{cases}.\qquad\ \ \ $$

\emph{Let $\hat{\G}_+(f)$ denote the convex hull of the set}
$$\bigcup_{\alpha \in V_{\G_+(f)}} \{\hat{\alpha} + \RR^n_+\}.$$
\emph{Then $\hat{\G}_+(f)$ is a polyhedron in $\RR^n_+.$}

\emph{Let $V_{\hat{\G}_+(f)}$ be the set of all vertices of $\hat{\G}_+(f).$}

\emph{For $\alpha^* \in V_{\hat{\G}_+(f)},$ put}
$$\Delta(\alpha^*):= \{\x=(x_1, ..., x_n) \in \RR^n: x_j = 0 \ \forall j \notin I(\alpha^*) \mbox{ and } x_i = t \ \forall i \in I(\alpha^*), \ t \in \RR\}.$$ 

\emph{Let $D(\alpha^*)$ denote the intersection point of $\Delta(\alpha^*)$ and the boundary of $\G_+(f),$ and $d(\alpha^*)$ denote the non-zero coordinate of $D(\alpha^*).$ }

\emph{Finally, we put }
$$d(f): = \max\{d(\alpha^*): \alpha^* \in V_{\hat{\G}_+(f)} \}.$$

\end{Notation}

\begin{Theorem}\label{Thm4.3.2} Let $f$ be a smooth function on a neighbourhood of the origin in $\RR^n,$ $f(0)=0$ and $\nabla f(0)=0.$ Assume that $f$ is non-negative, non-degenerate and satisfies the KN-condition. Then we have
$$\alpha(f) = d(f).$$
\end{Theorem}

Put $\hat{\ZZ}_+^n:=\{ \alpha=(\alpha_1, ..., \alpha_n) \in \ZZ^n_+, \alpha\ne 0 \mbox{ and } \alpha_i \in \{0,1\}\ \forall i=1,...,n\}.$

Let $\A$ be a finite subset of $\hat{\ZZ}_+^n.$

Let $X(\A)$ denote the germ at $0 \in \RR^n$ of the set $\{ \x \in \RR^n: \x^\alpha = 0 \ \forall \alpha \in \A\}$ and $ \langle \x^\alpha : \alpha \in \A\rangle \E_n$ denote the ideal generated by the monomials $\x^\alpha$, $\alpha \in \A$, in the ring $\E_n.$

\begin{Lemma}\label{Lem4.3.3} Let $\A$ be a finite subset of $\hat{\ZZ}_+^n$ and $h \in \E_n.$ Assume that $h(\x)$ vanishes on $X(\A)$. Then we have
$$h \in \langle \x^\alpha: \alpha \in \A \rangle \E_n.$$
\end{Lemma} 
\begin{proof}
Put $\displaystyle I(\A):= \bigcup_{\alpha \in \A} I(\alpha).$ 

Let $r(\A)$ denote the number of elements of $I(\A).$ We will prove Lemma \ref{Lem4.3.3} by induction on $r(\A).$

Assume that $r(\A) = 1$ and $h$ is an arbitrary smooth function, $h(0) =0$ for all $\x \in X(\A).$

Since $r(\A)=1,$ $\A= \{ \e_{i_0}\}$ for some $i_0 \in \{1,..., n\}$, where $\e_{i_0}= (0,...,0,\underset{i_0}{1},0,...,0).$

Hence $X(\A) = \{ \x \in \RR^n: x_{i_0} =0\}.$

Since $h$ vanishes on $X(\A)$, $h = x_{i_0}h_1,$ for some $h_1 \in \E_n$ and the lemma is true.

Next we assume that for any subset $\A'$ of $\hat{\ZZ}^n_+$ such that $r(\A') < r(\A),$ the following statement holds: if $h' \in \E_n$ and $h'$ vanishes on $X(\A')$, then $$h'\in \langle \x^\alpha: \alpha \in \A' \rangle \E_n.$$ We will show that, if $ h \in \E_n$ and $h$ vanishes on $X(\A)$, then $h \in \langle \x^\alpha: \alpha \in \A \rangle \E_n.$

For $\alpha=(\alpha_1, ..., \alpha_n) \in \A$, let
$$i(\alpha) := \min\{i : \alpha_i \ne 0\}.$$

For $i \in \{1,...,n\},$ put 
$$\A_i:=\{ \alpha \in \A: i(\alpha)=i\}.$$

Let $\{i_1, ..., i_s\}$, $i_1 < ... < i_s,$ be the set of all $i \in \{1,...,n\}$ for which $\A_i \ne \emptyset.$ Then
$$\A= \bigcup_{j=1}^s \A_{i_j}.$$

Let $h \in \E_n$ and $h$ vanishes on $X(\A).$ Let $h_0$ denote the restriction of $h$ on the hyperplane $\{ \x \in \RR^n: x_{i_1}=0\}.$ Then $h_0 \in \E_n$ and $h_0$ does not depend on the variables $x_{i_1}.$ Moreover, we have
$$h = h_0 + x_{i_1}\varphi,$$
with $\varphi \in \E_n.$

Put $\A ' =\A \setminus \A_{i_1}.$ We will show that $h_0$ vanishes on $X(\A ').$

By contradiction, assume that $h_0(\x_0) \ne 0$ for some $\x_0 \in X(\A ')$. Put $\tilde{\x}_0 = (\tilde{x}_{0,1}, ..., \tilde{x}_{0,n}),$ where
$$\tilde{x}_{0,i}= \begin{cases} x_{0,i} \mbox{ if } i \ne i_1\\ 0 \mbox{ if } i=i_1 \end{cases}.$$
Clearly, $\tilde{\x}_0 \in \A.$ Since $h_0$ does not depend on $x_{i_1},$ we have
$$0 = h(\tilde{\x}_0)= h_0(\tilde{\x}_0)= h_0({\x}_0) \ne 0,$$
which is a contradiction. Hence $h_0$ vanishes on $X(\A')$. Since $r(\A') < r(\A),$ by the hypothesis of induction, we have
\begin{equation}\label{4.17} h_0 \in \langle \x^\alpha : \alpha \in \A '\rangle \E_n.
\end{equation}

We see that 
\begin{itemize}
\item[(i)] If $\varphi =0$ then by (\ref{4.17})
$$h= h_0 \in \langle \x^\alpha : \alpha \in \A' \rangle \E_n \subset \langle \x^\alpha : \alpha \in \A \rangle \E_n,$$
and the conclusion of the lemma is true.
\item[(ii)] If $\varphi \ne 0,$ but $\e_{i_1} = (0,...,0,\underset{i_1}{1},0,...,0) \in \A_{i_1},$ then, again, by (\ref{4.17}), we have
$$h= h_0 + x_{i_1}\varphi \in \langle \x^\alpha : \alpha \in \A \rangle \E_n.$$ 
\end{itemize} 
So it remains to show that $h\in \langle \x^\alpha : \alpha \in \A \rangle \E_n$ for the case when $\varphi \ne 0$ and $\e_{i_1} \notin \A_{i_1}.$

For $\alpha = (\alpha_1, ..., \alpha_n) \in \A_{i_1},$ we define 
$\alpha_0:= (\alpha_{0,1}, ..., \alpha_{0, n}),$ where
$$\alpha_{0,i}= \begin{cases} \alpha_i, \mbox{ if } i \ne i_1\\ 0, \mbox{ if } i=i_1 \end{cases}.$$
Since $\e_{i_1} \notin \A_{i_1}$, clearly, $\alpha_0 \in \hat{\ZZ}_+^n\ \forall \alpha \in \A_{i_1}.$

Put $\B := \{\alpha_0 : \alpha \in \A_{i_1}\}.$ Then $h$ vanishes on the set $X(\B \cup \A '),$  since $h$ vanishes on $X(\A)$ and $X(\B) \subset X(\A_{i_1}).$ 

Since $r(\B \cup \A ') < r(\A)$, by the hypothesis of induction, we have
$$h \in \langle \x^\alpha : \alpha \in \B \cup \A '\rangle \E_n$$
or 
$$h = \sum_{\beta \in \B} \x^{\beta} \varphi_\beta + \sum_{\alpha ' \in \A '} \x^{\alpha '} \psi_{\alpha '},$$
where $\varphi_\beta$ and $\psi_{\alpha '}$ belong to $\E_n$ $\forall \beta \in \B$ and $\alpha'\in \A'.$

If $\beta \in \B$ then $\beta = \alpha_0$ for some $\alpha \in \A_{i_1},$ we can rewrite the above equality as follows
$$h = \sum_{\alpha \in \A_{i_1}} \x^{\alpha_0} \varphi_\alpha + \sum_{\alpha ' \in \A '} \x^{\alpha '} \psi_{\alpha '}.$$

Let $\x(t) = (x_1, ..., x_{i_1-1}, tx_{i_1}, x_{i_1+1},..., x_n).$ We have
\begin{align*}\displaystyle h(\x) - h_0(\x) &= \int_0^1 \sum
_{\alpha \in \A_{i_1}} \left(\x^{\alpha_0} \frac{d}{dt}\varphi_\alpha(\x(t))\right)dt + \int_0^1 \sum_{\alpha ' \in \A '} \left(\x^{\alpha '} \frac{d}{dt} \psi_{\alpha '}(\x(t)) \right)dt\nonumber\\
 &= \sum_{\alpha \in \A_{i_1}} \x^{\alpha} \hat{\varphi}_\alpha(\x) + \sum_{\alpha ' \in \A '} \x^{\alpha '} \hat{\psi}_{\alpha '}(\x),
\end{align*}
where 
$$ \hat{\varphi}_\alpha = \int_0^1 \frac{\partial \varphi_\alpha(\x(t))}{\partial x_{i_1}}dt, \ \forall \alpha \in \A_{i_1}$$
and 
$$\hat{\psi}_{\alpha '}= \int_0^1 x_{i_1}\frac{\partial \psi_{\alpha '}(\x(t))}{\partial x_{i_1}}dt, \ \forall \alpha'\in \A'.$$
Since $\hat{\varphi_\alpha} \in \E_n, \ \forall \alpha \in \A_{i_1}$ and $\hat{\psi}_{\alpha '} \in \E_n, \ \forall \alpha ' \in \A ',$ we have 
$$h-h_0 \in \langle \x^\alpha : \alpha \in \A \rangle \E_n.$$
Then, by (\ref{4.17}), we have
$$h\in \langle \x^\alpha : \alpha \in \A \rangle \E_n.$$
The lemma is proved.
\end{proof}

\noindent
{\bf Proof of Theorem \ref{Thm4.3.2}}

By Lemma \ref{Lem4.1.3}, we have
$$f \asymp g_{\G_+(f)}.$$
Hence, 
$$\left(f^{-1}(0),0\right)= \left(g_{\G_+(f)}^{-1}(0),0\right)$$
where $\left(f^{-1}(0),0\right)$ and $\left(g_{\G_+(f)}^{-1}(0),0\right)$ denote the germs respectively of $f^{-1}(0)$ and $g_{\G_+(f)}^{-1}(0)$ at the origin.
Therefore,
\begin{align*} (\left(f^{-1}(0),0\right) &= \{\x: \RR^n: |\x|^\alpha = 0 \ \forall \alpha \in V_{\G_+(f)}\}\\ &=\{\x: \RR^n: |\x|^{\hat{\alpha}} = 0 \ \forall \alpha \in V_{\G_+(f)}\}
\end{align*}
where $\hat{\alpha} = (\hat{\alpha}_1,...,\hat{\alpha}_n),$ $\hat{\alpha}_i =1$ if $i \in I(\alpha)$ and $\hat{\alpha}_i=0$ if $i \notin I(\alpha).$

By Lemma \ref{Lem4.1.1}, we have
$$\displaystyle \sum_{\alpha^* \in V_{\hat{\G}_+(f)}} |\x^{\alpha^*}| \asymp \sum_{\alpha \in V_{\G_+(f)}} |\x^{\hat{\alpha}}|.$$
From this it follows that
\begin{equation}\label{4.18}
(\left(f^{-1}(0),0\right) = \{\x: \RR^n: |\x|^{\alpha^*} = 0 \ \forall \alpha^* \in V_{\hat{\G}_+(f)}\}
\end{equation}

\noindent
{\it Proof of $\alpha(f) \le d(f):$}

\noindent
Claim: For any $\alpha^* \in V_{\hat{\G}_+(f)},$ there exist $c, \alpha >0$ such that
$$g_{\G_+(f)}(\x) \ge c||\x^{\alpha^*}||^{d(\alpha^*)},\ \forall ||\x|| < \varepsilon.$$

Proof: Let $\alpha^* \in V_{\hat{\G}_+(f)}.$ Then $\alpha^* \in \D(\alpha^*)$ and $D(\alpha^*) = d(\alpha^*) \alpha^* \in \G_+(f).$

In order to prove the claim, we will use the toric modification
$$\pi: Y_{\G_+(f)} \rightarrow \RR^n.$$
Let $\sigma = \RR_+\a^1(\sigma) + ...+ \RR_+\a^n(\sigma) \in \Sigma^{(n)}$ and $\y^0 \in \pi^{-1}(0).$ Then, it is easy to see that there exists a neighbourhood of $\y^0$ in which we have
$$ g_{\G_+(f)}\circ \pi_\sigma(\y_\sigma) \asymp \Big|y_{\sigma,1}^{l(\a^1(\sigma))}...y_{\sigma,n}^{l(\a^n(\sigma))}\Big|$$
and
$$\Big|x^{\alpha^*}\Big|^{d(\alpha^*)} \asymp \Big|y_{\sigma,1}^{L_1(\sigma)}...y_{\sigma,n}^{L_n(\sigma)}\Big|,$$
where $L_i(\sigma) = \langle \a^i(\sigma), d(\alpha^*)\alpha^*\rangle.$

Since $d(\alpha^*) \alpha^* \in \G_+(f),$ $L_i(\sigma) \ge l(\a^i(\sigma)) \ \forall i = 1,...,n.$

Hence, in a neighbourhood of the point $\y^0$, we have
$$|\x^{\alpha^*}|^{d(\alpha^*)} \circ \pi_\sigma(\y_\sigma) \in \O(g_{\G_+(f)} \circ \pi_\sigma(\y_\sigma)).$$ 
 
Since this relation holds for any $\sigma \in \Sigma^{(n)}$ and any $\y^0 \in \pi^{-1}_\sigma(0),$ we conclude that in some neighbourhood of $0 \in \RR^n$, we have
$$|\x^{\alpha^*}|^{d(\alpha^*)} \in \O(g_{\G_+(f)}).$$
Hence, the claim is proved.

Now, let $h$ be a smooth function on a neighbourhood of the origin and $f^{-1}(0) \subset h^{-1}(0).$ Then, by Lemma \ref{Lem4.3.3} and (\ref{4.18}), $h$ can be written in the form
$$h(\x) = \sum_{\alpha^* \in V_{\hat{\G}_+(f)}}\x^{\alpha^*}\varphi_{\alpha^*}(\x),$$
where $\varphi_{\alpha^*}\in \E_n$ for every $\alpha^* \in V_{\hat{\G}_+(f)}.$

This implies that
$$|h| \in \O\left( \sum_{\alpha^* \in V_{\hat{\G}_+(f)}} |\x^{\alpha^*}| \right)$$

Hence, by the above claim,
$$\displaystyle |h| \in \O\left( \sum_{\alpha^* \in V_{\hat{\G}_+(f)}}g^{\frac{1}{d(\alpha^*)}}_{\G_+(f)} \right).$$
And, since $d(f) =\max \{d(\alpha^*):$ $\alpha^* \in V_{\hat{\G}_+(f)}\},$ we get
$$|h| \in \O\left(g^{\frac{1}{d(f)}}_{\G_+(f)}\right),$$
or, $$|h|^{d(f)} \in \O \left(g_{\G_+(f)}\right).$$
Because $f \asymp g_{\G_+(f)}$ (Lemma \ref{Lem4.1.3}), we obtain 
$$|h|^{d(f)} \in \O(f),$$
which implies that $\alpha(f) \le d(f).$

\noindent
{\it Proof of the inequality $\alpha(f) \ge d(f):$}

By contradiction, assume that $\alpha(f) < d(f).$

Let $\alpha^*$ denote the element of $V_{\hat{\G}_+(f)}$ such that $d(\alpha^*) = d(f).$ Clearly, $\x^{\alpha^*}$ vanishes on $(f^{-1}(0),0).$ Since $\alpha(f) <d(f),$ the point $\alpha(f)\alpha^*$ does not belong to $\G_+(f).$ Then there exist $\sigma=\RR_+\a^1(\sigma) +...+ \RR_+\a^n(\sigma) \in \Sigma^{(n)}$ and $i_0 \in \{1,...,n\}$ such that 
$$\langle \a^{i_0}(\sigma), \alpha(f) \alpha^* \rangle < l(\a^{i_0}(\sigma)).$$
In any small neighbourhood $U$ of the point $o \in \RR^n(\sigma)$, we have 
$$g_{\G_+(f)} \circ \pi_\sigma(\y_\sigma) \asymp \Big|y_{\sigma,1}^{l(\a^1(\sigma))}\ldots y_{\sigma,n}^{l(\a^n(\sigma))}\Big|$$
$$|\x^{\alpha(f)\alpha^*}| \circ \pi_\sigma(\y_\sigma) = \Big|y_{\sigma,1}^{L_1}\ldots y_{\sigma,n}^{L_n}\Big|,$$
where $L_i = \alpha(f)\langle \a^i(\sigma), \alpha^* \rangle.$

Since $L_{i_0} < l(\a^{i_0}(\sigma)),$ the relation
 $$|\x^{\alpha(f)\alpha^*}| \circ \pi_\sigma(\y_\sigma) \in \O\left(g_{\G_+(f)} \circ \pi_\sigma(\y_\sigma)\right)$$
 does not hold on $U$. Hence, the relation
$$|\x^{\alpha^*}|^{\alpha(f)} \in \O(g_{\G_+(f)})$$
does not hold on the set $\pi_\sigma(U).$ Now, since $f \asymp g_{\G_+(f)}$ and $\pi_\sigma(U) \cap \mathbb{B}_\varepsilon \ne 0 $ for any $\varepsilon >0,$ the inequality
$$f(\x) \ge c|\x^{\alpha^*}|^{\alpha(f)},  \forall ||\x|| \le \varepsilon$$
does not hold for any $c>0$ and $\varepsilon>0.$ This contradiction shows that $\alpha(f) \ge d(f)$ and the theorem is proved. \hfill $\Box$

\subsection{Computation of $\mathcal{L} (f)$}$\ $

We are going to compute $\mathcal{L}(f)$ under the following condition: $f$ is smooth on a neighbourhood of the origin, $f(0)=0, \ \nabla f(0)=0;$ $f$ is non-negative, non-degenerate, and satisfying the KN--condition.

With the notation of \S 4.3, we put $$I(f) := \displaystyle \bigcup_{\alpha^*\in V_{\hat{\G}_+(f)}} I(\alpha^*)$$ 

Let $I(f) = \{i_1 <...<i_s\}.$ Let $S_{I(f)}$ denote the set of all bijections of   $I(f)$. 

For $\rho \in S_{I(f)},$ we put
$$U(\rho):= \{\x = (x_1,...,x_n) \in \RR^n: |x_{\rho(i_1)}| \le ...\le |x_{\rho(i_s)}|\}.$$

Since $f$, by Theorem \ref{Theo3.1}, admits the inequality (\L$_2$), there exist $c, \varepsilon >0$ and $\tau >0$ such that
\begin{equation}\label{4.19} f(\x) \ge c \dist (\x, f^{-1}(0))^\tau,\ \forall ||\x|| \le \varepsilon, \x \in U(\rho).
\end{equation}

Put, for $\rho \in S_{I(f)},$
$$\mathcal{L}_\rho (f) = \inf \{ \tau:  \mbox{ (\ref{4.19}) holds}\}.$$

Let $J$ be a subset of $I(f)$. We say that {\it $J$ satisfies the condition $(*)$}, if for any $\alpha^* \in V_{\hat{\G}_+(f)},$ the set $J \cap I(\alpha^*)$ consists of a unique element.

Let $\Lambda$ denote the set of all subset $J$ of $I(f)$ satisfying the condition $(*)$. By (\ref{4.18}), we have
\begin{equation}\label{4.20}
\left(f^{-1}(0),0\right) = \left( \bigcup_{J \in \Lambda} T_J(\RR^n),0\right).
\end{equation} 

To compute $\Lc_\rho(f),$ it is more convenient to use, instead the Euclidean norm, the max-norm:
$$||(x_1, ..., x_n)||_{\max} = \max\{ |x_i|: i =1, ..., n\}.$$

Let $J \in \Lambda.$ Then
$$d(\x, T_J(\RR^n) = \max\{|x_i| : i \in J\},$$
where $\x \in \RR^n$ and $d(\x, T_J(\RR^n) $ denotes the distance from $\x$ to $T_J(\RR^n)$ in the max-norm. 

If, in addition, $\x \in U(\rho),$ then
$$d(\x, T_J(\RR^n)= |x_{i(J,\rho)}|,$$
where $i(J, \rho)$ denotes the element of $J$ such that
$$\rho(i(J, \rho)) = \max\{\rho(i): i \in J\}.$$

For $\rho \in S_{I(f)}$, we denote by $i_\rho$ the element of $\{ i(J, \rho): J \in \Lambda\}$ such that
$$\rho(i_\rho) = \min \{\rho(i(J, \rho)): J \in \Lambda\}.$$

Then, by (\ref{4.20}), for any $\x \in U(\rho)$, sufficiently close to the origin, we have
\begin{equation}\label{4.21}
d(\x, f^{-1}(0)) = |x_{i_\rho}|.
\end{equation}

For $\rho \in S_{I(f)}$, we put
$$V_{\G_+(f)}(\rho):= \{\alpha \in V_{\G_+(f)}: I(\alpha) \subset I(f) \mbox{ such that } \forall i \in I(\alpha): \rho(i) \ge \rho(i(\rho))\}.
$$

\begin{Theorem}\label{Theo4.4.2} Let $f$ be a smooth function on a neighbourhood of the origin in $\RR^n,$ $f(0)=0,$ and $\nabla f(0) =0.$ Assume that $f$ is non-degenerate, non-negative and satisfies the KN-condition. Then the following statements hold true
\begin{enumerate}[label=\upshape(\roman*), ref=\theTheorem.(\roman*)]
\item\label{4.4.2i} For any $\rho \in S_{I(f)},$ the set $V_{\G_+(f)}(\rho)$ is not empty;
\item\label{4.4.2ii} For any $\rho\in S_{I(f)},$
$$\Lc_\rho(f) = \min \{|\alpha| : \alpha \in V_{\G_+(f)}(\rho)\}.$$
\item\label{4.4.2iii} $\displaystyle \Lc(f) = \max_{\rho \in S_{I(f)}} \Lc_\rho(f).$
\end{enumerate}
\end{Theorem}
\begin{proof}$ \ $

\noindent
{\it Proof of (i):} By contradiction, assume that $V_{\G_+(f)}(\rho)= \emptyset$ for some $\rho \in S_{I(f)}.$ Then it is not difficult to show that for any $\alpha^* \in V_{\hat{\G}_+(f)}$ there exists $i(\alpha^*) \in I(\alpha^*)$ such that 
$$\rho(i(\alpha^*)) < \rho(i_\rho).$$

Put
$$J^*:=\{ i_{\alpha^*}: \alpha^* \in V_{\hat{\G}_+(f)}\}.$$
We have then
$$d(\x, T_{J^*}(\RR^n)) = \max\{|\x_{i(\alpha^*)}| : \alpha^* \in V_{\hat{\G}_+(f)}\}$$
for all $\x \in U(\rho).$

The above inequality show implies that, for $\x \in U(\rho)$
$$d(\x, T_{J^*}(\RR^n))< |x_{i_\rho}|,$$
which contradicts to (\ref{4.21}). Thus, (i) is proved.\\

\noindent
{\it Proof of (ii):} Let $\alpha = (\alpha_1, ..., \alpha_n) \in V_{\G_+(f)}$. Then, as in the proof of Lemma \ref{Lem4.1.3}, $\alpha_i$ is even number for any $i=1,...,n.$ We have
$$g_{\G_+(f)}(\x) = \sum_{\alpha \in V_{\G_+(f)}} |\x^\alpha| = \sum_{\alpha \in V_{\G_+(f)}} \x^\alpha.$$
Hence
$$g_{\G_+(f)}(\x) \ge \sum_{\alpha \in V_{\G_+(f)}(\rho)} \x^\alpha.$$

It is not difficult to see that, if $\x \in U(\rho)$ then
$$g_{\G_+(f)}(\x)  \ge \sum_{\alpha \in V_{\G_+(f)}(\rho)}|x_{i_\rho}|^{|\alpha|}.$$
Hence, 
\begin{equation}\label{4.22}g_{\G_+(f)}(\x)  \ge |x_{i_\rho}|^{|\alpha^0|},
\end{equation}
where $\alpha^0$ denotes the element of $V_{\G_+(f)}(\rho)$ such that 
$$|\alpha^0| = \min\{ |\alpha|: \alpha \in V_{\G_+(f)}(\rho)\}.$$

Now, by Lemma \ref{Lem4.1.3}, we have
$$f(\x) \asymp g_{\G_+(f)}(\x).$$
This , together with (\ref{4.21}) and (\ref{4.22}) show that
$$\Lc_\rho(f) \le |\alpha^0| = \min\{ |\alpha|: \alpha \in V_{\G_+(f)}(\rho)\}.$$

Now we will prove that $\Lc_\rho(f) \ge \min\{ |\alpha|: \alpha \in V_{\G_+(f)}(\rho)\}.$

Put $\displaystyle I(g_{\G_+(f)}) = \bigcup_{\alpha \in V_{\G_+(f)}} I(\alpha)$

$$J(\rho):= \{ i\in I(f): \rho(i) < \rho(i_\rho)\}.\qquad \qquad \qquad \qquad \qquad \qquad \qquad \quad$$
$$X:= \left\{ (x_1, ..., x_n) \in U(\rho): x_j = 0 \ \forall j \in J(\rho) \cup (I(g_{\G_+(f)}) \setminus I(f))\right\}.\quad $$
$$Y:= \left\{ (x_1, ..., x_n) \in U(\rho): x_j = x_{_{i_\rho}} \ \forall j \in I(f) \mbox{ such that } \rho(i) \ge \rho(i_\rho)\right\}.$$
Restricting $g_{\G_+(f)}(\x)$ on the set $X\cap Y$, we see that
$$g_{\G_+(f)}(\x) \Big|_{X\cap Y} \asymp |x_{i_\rho}|^{|\alpha^0|}.$$
This shows that 
$$\Lc_\rho(f) \ge |\alpha^0| = \min \{|\alpha|:\alpha \in V_{\G_+(f)}(\rho)\},$$
and (ii) is proved.\\

\noindent
The proof of (iii) is trivial, since 
$$\displaystyle \RR^n = \bigcup_{\rho \in S_{I(f)}}U(\rho).$$

The theorem is proved.
\end{proof}

\subsection{A consequence of Theorems \ref{Thm4.2.3}, \ref{Thm4.3.2} and \ref{Theo4.4.2}}

In this part, we combine the results of \S 4.2. -- \S 4.4. and introduce a (very special) class of smooth functions, all the \L ojasiewicz exponents of which can be computed explicitly and simply. 

Let us recall the definition of partially convenientness. A function $f$ is partially convenient, if there exist $I \subset \{1, ..., n\}$ and positive integers $\nu_i,$ $i \in I$ such that
\begin{itemize}
\item[(a)] $\{\nu_i \e_i, i \in I\} \subset V_{\G_+(f)}$ and
\item[(b)] $I(f) =I.$
\end{itemize}
If $f$ is partially convenient, then let $\nu(f)$ denote the maximum of $\nu_i,$ $i \in I.$
\begin{Theorem}\label{Theo4.5.1} Let $f$ be a smooth function on a neighbourhood of the origin, $f(0)=0,$ and $\nabla f(0) =0.$ Assume that $f$ is non-negative, non-degenerate, partially convenient and satisfies the KN-condition. Then we have
\begin{itemize}
\item[(i)] $\theta(f) = 1-\displaystyle \frac{1}{\nu(f)}$;
\item[(ii)] $\alpha(f) = \nu(f);$
\item[(iii)] $\Lc(f) = \nu(f).$
\end{itemize}
\end{Theorem}
\begin{proof}
(i) follows from Theorem \ref{Thm4.2.3}.

\noindent
{\it Proof of (ii):} Let $I \subset \{1, ..., n\}$ and $\{\nu_i : i \in I\}$ be as in the definition of partially convenientness. Then, since $\nu_i \e_i \in \G_+(f),$ $\forall i \in I$ and $I(f) =I,$ it is easy to see that 
$$V_{\hat{\G}_+(f)}= \{\e_i , i \in I\}$$
$$\nu_i = d(\e_i)$$
$$\mbox{and } \nu(f) = d(f).$$
Hence, since $f$ is non-negative, non-degenerate and satisfies the KN-condition, $\alpha(f) = d(f) = \nu(f)$ by Theorem \ref{Thm4.3.2}. Thus (ii) is true.

\noindent
{\it Proof of (iii):} It follows from Theorem \ref{Theo4.4.2} that
$$\Lc(f) \le \max\{|\alpha|: \alpha \in V_{\G_+(f)}\}.$$
Since $f$ is partially convenient, we have
$$\max\{|\alpha|: \alpha \in V_{\G_+(f)}\} = \nu(f).$$
Hence, $\Lc(f) \le \nu(f).$

Since $V_{\hat{\G}_+(f)}=\{ \e_i, i \in I\},$
$$(f^{-1}(0),0) = (T_I(\RR^n),0),$$
where $T_I(\RR^n) =\{ \x=(x_1, ..., x_n) \in \RR^n: x_i =0 \ \forall i \in I\}.$

Let $\nu(f) = \nu_{i_0}$, $i_0 \in I.$

Restricting $g$ and $g_{\G_+(f)}$ on the set
$$X= \{ \x = (x_1, ..., x_n) \in \RR^n: x_i =0 \ \forall i \in I \mbox{ and } i \ne i_0\}.$$
We see that
$$f\Big|_X \asymp g\Big|_X \asymp x_{i_0}^{\nu(f)}$$
and 
$$\dist(\x, f^{-1}(0)) = \dist (\x, T_I(\RR^n)) = |x_{i_0}|$$
for $\x \in X.$

This shows that $\Lc(f) \ge \nu(f)$ and therefore, (iii) is proved.
\end{proof}
\section{The case of convex functions}
We begin this section with the following
\begin{LemmaS}\label{Lem5.1} Let $f$ be a smooth convex function on a neighbourhood of the origin, $f(0) =0,$ and $\nabla f(0)=0.$ Assume that $f$ satisfies the KN-condition, then there exist $I \subset\{1,...,n\}$ and positive even numbers $\nu_i$, $i \in I$ such that 
$$V_{\G_+(f)} = \{\nu_i \e_i : i \in I\}.$$
\end{LemmaS}
\begin{proof}
Since $f$ satisfies the KN-condition, it can be written in the form
$$f= \sum_{\rho \in S}\x^\rho \varphi_\rho$$
where $S$ is a finite subset of $\G_+(f) \cap \ZZ^n_+,$ and $\varphi_\rho \in \E_n,$ $\forall \rho \in S.$

Now, let $\alpha \in V_{\G_+(f)}.$ Then, clearly $\alpha \in S$ and $\varphi_\alpha(0) \ne 0.$ We can rewrite $f$ in the form
$$f= c\x^\alpha + \sum_{\rho' \in S'} \x^{\rho'} \varphi_{\rho'}$$
where $c \ne 0,$ $S'$ is a finite subset of $\G_+(f) \cap \ZZ^n_+$ and $\varphi_\alpha(0) \in \E_n, \ \forall \rho'\in S'.$

Put $P(\x) : = c\x^\alpha$
$$h(\x): = \sum_{\rho' \in S'} \x^{\rho'} \varphi_{\rho'}.$$
We will prove that $P(\x)$ is a convex function. By contradiction, assume that it is not the case, then there exist $\x_0 =(x_{01}, ..., x_{0n}) \in \RR^n$ and $\v^0 = (v^0_1,..., v^0_n) \in \RR^n$,  $||\v^0||=1$ such that
\begin{equation}\label{5.1}
\sum_{i,j=1}^n \frac{\partial^2 P(\x_0)}{\partial x_i \partial x_j}v^0_i v^0_j <0.
\end{equation}

Since $\alpha \in V_{\G_+(f)},$ there exist $\a=(a_1,..., a_n) \in \ZZ^n_+$ and $l \in \ZZ_+$ such that 
$$\langle \a, \alpha\rangle = l$$
and 
$$\langle \a, \x \rangle > l\ \ \forall \x \in \G_+(f), \x \ne \alpha.$$
Put $\x_0(t) = (t^{a_1}x_{01}, ..., t^{a_n}x_{0n}).$

Then 
$$\displaystyle  \sum_{i,j=1}^n \frac{\partial^2 P(\x_0(t))}{\partial x_i \partial x_j}(t^{a_i}v^0_i)(t^{a_j} v^0_j) = t^l\left(\sum_{i,j=1}^n \frac{\partial^2 P(\x_0)}{\partial x_i \partial x_j}v^0_i v^0_j\right)$$
and
$$\displaystyle  \sum_{i,j=1}^n \frac{\partial^2 h(\x_0(t))}{\partial x_i \partial x_j}(t^{a_i}v^0_i)(t^{a_j} v^0_j) =o(t^l).$$
This follows that
$$\displaystyle  \sum_{i,j=1}^n \frac{\partial^2 f(\x_0(t))}{\partial x_i \partial x_j}(t^{a_i}v^0_i)(t^{a_j} v^0_j) \asymp t^l\left(\sum_{i,j=1}^n \frac{\partial^2 P(\x_0)}{\partial x_i \partial x_j}v^0_i v^0_j\right).$$
The left hand side is non-negative, since $f$ is convex, while the right hand side can be negative ( because of (\ref{5.1})) for $t$ sufficiently close to zero. The contradiction shows that 
$P(\x) = c\x^\alpha$ is a convex function.

Therefore, $P(\x) = c\x^\alpha$ must have the form 
$$P(\x) = c \x^{\nu_i \e_i},$$
for some $c>0,$ $\nu_i >0$ and even. Then, the lemma is proved. 
\end{proof}
\begin{TheoremS}\label{Thm5.2} Let $f$ be a smooth convex function on a neighbourhood of the origin, $f(0) =0,$ $\nabla f(0) =0$. Assume that $f$ is non-degenerate and satisfies the KN-condition. Then we have
\begin{itemize}
\item[(i)] $\theta(f) = 1- \displaystyle \frac{1}{\nu(f)};$
\item[(ii)] $\alpha(f) = \nu(f);$
\item[(iii)] $\Lc(f) = \nu(f).$
\end{itemize} 
\end{TheoremS}
\begin{proof}
Since $f$ is convex, $f(0) =0,$ $\nabla f(0) =0$, then $f$ is non-negative in a neighbourhood of the origin. Hence, the theorem follows from Lemma \ref{Lem5.1} and Theorem~\ref{Theo4.5.1}.
\end{proof}

We finish the paper with the following
\begin{RemarkS}
\begin{itemize}
\item[\upshape(i)]\emph{ All the results of Sections 4--5 are still true, if we replace ``$f$ is a smooth function satisfying the KN-condition" by ``$f$ is an analytic function", or ``$f$ belongs to the Denjoy-Carleman quasi-analytic classes", or ``$f$ is convenient".}
\item[\upshape(ii)]\emph{In our knowledge, the formulas for $\theta(f)$ and $\alpha(f)$, given respectively in Theorem \ref{Thm4.2.3} and Theorem \ref{Thm4.3.2} are new even for the case of analytic functions. The exponent $\Lc(f)$, where $f$ is a non-negative, non-degenerate analytic function, was computed firstly in \cite{BP}. Our method of computation, given in Therem \ref{Theo4.4.2}, is different from that of \cite{BP}.}
\end{itemize}
\end{RemarkS}

\end{document}